\documentclass[11pt]{article}
\usepackage{amsmath,amsfonts,mathtools,amssymb,amsthm,mathrsfs}

\usepackage{graphicx}
\usepackage{color}

\newcommand{\be}{\begin{equation}}
\newcommand{\ee}{\end{equation}}
\newcommand{\bea}{\begin{eqnarray}}
\newcommand{\eea}{\end{eqnarray}}
\newcommand{\bean}{\begin{eqnarray*}}
\newcommand{\eean}{\end{eqnarray*}}

\newcommand\blfootnote[1]{%
  \begingroup
  \renewcommand\thefootnote{}\footnote{#1}%
  \addtocounter{footnote}{-1}%
  \endgroup
}

\def\na{\nabla}
\def\div{\nabla\cdot}

\def\p{\partial}

\def\mn{|\!\!|}

\def\mn2{|\!\!|_{M^{d/2}}}

\newcommand{\R}{\mathbb R}

\renewcommand{\d}{\mathrm{d}}

\renewcommand{\div}{{\rm {div\ }}}

\def\<{\langle }

\renewcommand{\qed}{\qquad\kern1pt   %QED mark
   \vbox{\hrule height 0.6pt      %top
         \hbox{\vrule width 0.6pt %left
               \vbox{\vskip 6pt}  %skip
               \hskip 3pt
              \vrule width 1.3pt} %right
         \hrule depth 1.3pt}     %botom
   \kern1pt}

\renewcommand{\div}{{\rm {div\ }}}

\newcommand{\eqnb}{\begin{equation}}
\newcommand{\eqne}{\end{equation}}

\newtheorem{theorem}{Theorem}
\newtheorem{proposition}[theorem]{Proposition}
\newtheorem{lemma}[theorem]{Lemma}
\newtheorem{example}[theorem]{Example}

\newtheorem{cor}[theorem]{Corollary}

\theoremstyle{definition}
\newtheorem{definition}[theorem]{Definition}
\def\qed{\hfill $\square$}

\newtheorem{remark}[theorem]{Remark}
\makeatletter
    
    \@addtoreset{equation}{section}
  \makeatother
\usepackage{chngcntr}
\counterwithout{equation}{section}

\newcommand{\RR}{\mathbb{R}}
\newcommand{\T}{\mathbb{T}}

\begin{document}
\title{A partial uniqueness result and an asymptotically sharp nonuniqueness result for the Zhikov problem on the torus}
\author{
Tomasz Cie\'slak, Wojciech S. O\.za\'nski}
\date{\small\today}
\maketitle
\blfootnote{T.~Cie\'slak: Institute of Mathematics, Polish Academy of Sciences, Warsaw, 00-656, Poland; email: cieslak@impan.pl,\\
W.~S.~O\.za\'nski: Department of Mathematics, University of Southern California, Los Angeles, CA 90089, USA, and Institute of Mathematics, Polish Academy of Sciences, Warsaw, 00-656, Poland; email: ozanski@usc.edu}
\begin{abstract}
We consider the stationary diffusion equation $-\mathrm{div} (\nabla u + bu )=f$ in $d$-dimensional torus $\T^d$, where $f\in H^{-1}$ is a given forcing and $b\in L^p$ is a divergence-free drift. Zhikov \emph{(Funkts. Anal. Prilozhen., 2004)} considered this equation in the case of a bounded, Lipschitz domain $\Omega \subset \RR^d$, and proved existence of solutions for $b\in L^{2d/(d+2)}$, uniqueness for $b\in L^2$, and has provided a point-singularity counterexample that shows nonuniqueness for $b\in L^{3/2-}$ and $d=3,4,5$. We apply a duality method and a DiPerna-Lions-type estimate to show uniqueness of the solutions constructed by Zhikov for $b\in W^{1,1}$. We use a Nash iteration to demonstrate sharpness of this result, and also show that solutions in $H^1\cap L^{p/(p-1)}$ are flexible for $b\in L^p$, $p\in [1,2(d-1)/(d+1))$; namely we show that the set of $b\in L^p$ for which nonuniqueness in the class $H^1\cap L^{p/(p-1)}$ occurs is dense in the divergence-free subspace of $L^p$.
\end{abstract}
\section{Introduction}\label{first_introduction}
We consider $d$-dimensional periodic torus $\T^d=[-1/2,1/2]^d$, $d\geq 2$, and $b\in L^p (\T^d ;\RR^d)$ that is weakly divergence-free, that is
\begin{equation}\label{divergence}
\int_\Omega b\cdot \nabla \phi =0\;\;\mbox{for all}\;\; \phi \in C^\infty (\T^d ).
\end{equation}
We are concerned with the problem of Zhikov, that is with existence and uniqueness of solution $u\in \dot{H}^1 \coloneqq  \{ v\in H^1 (\T^d ) \colon \int v=0 \}$ to the stationary diffusion equation
\eqnb\label{zhikov_intro}
- \mathrm{div} (\nabla u + b u ) = f,
\eqne
where $f \in H^{-1} (\T^d )$ is given. Such a problem arises as a simplification of a homogenization problem, see \cite{FP}. We say that $u\in \dot{H}^1$ is a weak solution of \eqref{zhikov_intro} if 
\eqnb\label{zhikov_weak_intro}
\int ( -u\Delta \phi  + bu  \cdot \nabla \phi )  = (f,\phi )
\eqne
for all $\phi \in \dot C^\infty (\T^d )$. Here, and below, we use the notation $\int \equiv \int_{\T^d}$, $\| \cdot \|_p \coloneqq  \| \cdot \|_{L^p (\T^d )}$, we use the standard definitions of the Lebesgue spaces $L^p=L^p (\T^d )$ and the Sobolev spaces $W^{k,p}\equiv W^{k,p}(\T^d )$, $H^1 \equiv H^1(\T^d )$ and we use a dot ``$\cdot$'' to denote the subspace of functions with vanishing mean over $\T^d$, such as $\dot{L}^\infty$, for example.

Remarkably, despite the seemingly simple nature of \eqref{zhikov_intro}, it admits several surprising properties which were pointed out by Zhikov \cite{Zh} in the setting of a bounded, Lipschitz domain $\Omega$, and which we outline below. The purpose of this note is to present two results regarding uniqueness of solutions in the case of the torus $\T^d$: one of them gives uniqueness of a particular type of solutions (Theorem \ref{thm_main1}), and another one becomes sharp as $d\to \infty$ (Theorem \ref{thm_main2}).

In order to discuss the background of the problem, we note that so far it has been considered in the setting of a bounded, Lipschitz domain $\Omega \subset \RR^d$, in which case one uses analogous notions of the function spaces $L^p$, $W^{k,p}$, $H^1$, one considers the space $H^1_0 (\Omega )$ (the closure of $C_0^\infty (\Omega )$ with respect to the $H^1$ norm) instead of $\dot{H}^1$, one writes $\int\equiv \int_\Omega$ and one considers the weak formulation \eqref{zhikov_weak_intro} only for $\phi\in C_0^\infty$. Unless specified otherwise, the following results translate directly to the case of the torus $\T^d$.

First of all, if $b\in L^\infty$ then the bilinear form $\int bu \cdot \nabla \phi $ is bounded on $H^1_0$, and so existence and uniqueness of weak solutions  \eqref{zhikov_weak_intro} follow directly from the Lax-Migram lemma (see, for example, Theorem 6.2 in \cite{alt_book}). In that case, one observes two estimates that are independent of $b$: on the one hand taking $\phi \coloneqq  u$ gives the \emph{energy identity}
\eqnb\label{ee}
\int |\nabla u |^2 = (f,u).
\eqne
In fact, the energy identity is valid for any $u\in H^1_0$, $f\in H^{-1}$ and $b\in L^\infty$, since then $  \na |u|^2 \in L^1$, and so  approximating $u$ in $H^1$ by a $C^\infty_0$ functions allows one to take the limit in the divergence-free condition for $b$ to obtain \eqref{ee}.

On the other hand, if the forcing $f\in L^\infty$ (in which case we define $(f,u)\coloneqq \int fu$) then one has the \emph{maximum principle}
\eqnb\label{max_princ}
\| u \|_\infty \leq c_0 \| f \|_\infty,
\eqne
for some $c_0>0$ independent of $b$, which is stated in \cite[eq.~(1.5)]{Zh}; it can be verified using a Moser-type iteration. It is valid for all $u\in  H^1_0$, $b\in L^2$, $f\in L^\infty $ (and analogously in the periodic setting; \eqref{max_princ} is also valid for approximation solutions, see Definition~\ref{appr} below), which we prove in the Appendix for the sake of completeness. Some other questions related to regularity of solutions were also pursued in \cite{filonov, SSSZ}.

The problem is not so clear when $b\in L^p$ for $p < \infty$.

In fact, if $f\in L^\infty$ then one can easily construct a solution $u\in L^\infty$ to \eqref{zhikov_intro} for any $b\in L^1$ using the maximum principle \eqref{max_princ}. Indeed, it suffices  to approximate $b$ by a sequence of $b_n\in L^\infty$ such that $\| b_n -b \|_1 \to 0$, and extracting a weakly-$*$ convergent subsequence in $L^\infty$ of the corresponding (unique) solutions $u_n$ to obtain a desired solution $u$. Using the energy identity \eqref{ee} one can also make sure that the constructed $u$ belongs to $H^1_0$ with $\| u \|_{H^1} \leq c_0 |\Omega |  \| f \|_\infty$. Such approximation procedure gives rise to the definition of approximation solutions.
\begin{definition}\label{appr}
Given a $f\in H^{-1}$, $p\in [1,\infty ]$ and $b\in L^p $ that is weakly divergence-free \eqref{divergence}, we say that $u\in L^{p'}$ is an \emph{approximation solution} of the Zhikov problem \eqref{zhikov_intro} if it satisfies \eqref{zhikov_intro} in the sense of distributions and there exists a sequence $\{ b_n \} \subset  L^\infty$ such that $\| b_n - b\|_{p} \to 0$ and $u_n \rightharpoonup u$ weakly in $L^{p'}$ as $n\to \infty$, where $u_n$ denotes the unique solution of the Zhikov problem \eqref{zhikov_weak_intro} with $b$ replaced by $b_n$.
\end{definition}
Here and below $p'$ denotes the conjugate exponent to $p$ (i.e. $1/p'+1/p=1$). As mentioned above, if $f\in L^\infty$ then any approximation solution satisfies the maximum principle \eqref{max_princ}.

On the other hand, in the case when $f\not \in L^\infty$ then we have no reason to expect solutions in $L^\infty$, and so, in light of the embedding $H^1 \subset L^{2d/(d-2)}$, in order to make sense of the term $\int bu\cdot \nabla \phi$ in the weak formulation \eqref{zhikov_weak_intro} we assume that
\eqnb\label{b_approximate}
b\in L^\frac{2d}{d+2}.
\eqne
For such $b$ any approximation solution is a weak solution satisfying the energy inequality,
\eqnb\label{en_ineq}
\int |\nabla u |^2 \leq (f,u),
\eqne
which is a consequence of properties of weak limits. Given $b$ satisfying \eqref{b_approximate}, the question whether any weak solution satisfying \eqref{en_ineq} must be an approximation solution remains open and is related to the question of uniqueness of weak solutions. \\

As for uniqueness, Zhikov \cite{Zh} uses approximation by truncation to show uniqueness of weak solutions of \eqref{zhikov_intro}, in any dimension $d$, for
\eqnb\label{b_l2}
b\in L^2.
\eqne
Such uniqueness result is also true in the case of the torus in the class of functions $u\in \dot{H}^1$, which can also be proved by truncation and observing that \eqref{zhikov_weak_intro} is invariant with respect to adding constants to $u$.\\

In the case of $d=3$, he also provided an elegant example of nonuniqueness for $b\in L^{3/2-}$.
\begin{example}[Zhikov's \cite{Zh} counterexample]\label{ex_zhikov}
Let $\Omega \coloneqq  B(1)\subset \R^3$ and
\[
b(x)\coloneqq  \frac{x}{|x|^3} \beta \left( \frac{x}{|x|} \right),\qquad v(x)\coloneqq  (1-|x|^4 ) \alpha \left( \frac{x}{|x|}\right),
\]
where $\alpha, \beta \in C^\infty (S^2)$ are such that $\int _{S^2} \beta =0$, $\int_{S^2} \alpha \beta  =0 $ and $\int_{S^2} \alpha^2 \beta =-2$. Then $\div b=0$, $b\in L^{3/2-\epsilon}(\Omega )$ for every $\epsilon\in (0,1/2]$, $f\coloneqq  -\div (\nabla v + bv)\in H^{-1} (\Omega )$ and $v$ is a solution to the Zhikov problem that violates the energy inequality, that is $\int |\nabla v |^2 >(f,v)$. In particular, $v\ne u$ for any approximation solution $u$.
\end{example}
\begin{remark}
One can extend this example to dimension $4$ and $5$ by extending the three-dimensional ball into a cylinder. In dimension 6 and higher the existence condition \eqref{b_approximate} fails, and in dimension $2$ the example fails, as then $v\not \in H^1_0$.
\end{remark}

Example \ref{ex_zhikov} can be easily translated to the case of the torus $\T^d$, by cutting off $b$ outside of the singularity, and applying a Bogovski\u{\i} lemma (see \cite{bogovskii_79,bogovskii_80}, or Lemma III.3.1 in Galdi \cite{galdi_book} for details) to recover zero divergence, and then extending periodically.

The issue of uniqueness in the case $d=3$ for $b\in L^p$, $p\in (3/2,2)$, remains open. However, one can show uniqueness among approximation solutions in the case of $\Omega \coloneqq  B(1) \subset \RR^3$ for $b\in L^{3/2}$ that are smooth outside of the origin, see Lemma~2.2 in \cite{Zh}. Furthermore, as shown in Lemma~1.5 in
\cite{Zh}, one can use density of $L^\infty$ in $H^{-1}$ to obtain the following.
\begin{lemma}[Conditional uniqueness of approximation solutions]\label{zhikovs_lem1.5}
If $u=0$ is the only $\dot L^\infty $ solution of the Zhikov problem \eqref{zhikov_weak_intro} with $f=0$, then approximation solutions are unique in the class of all weak solutions.
\end{lemma}

The first main result of this note shows that, in the case of the torus $\T^d$, the condition of the above lemma holds if $b\in W^{1,1}$.

\begin{theorem}[Partial uniqueness]\label{thm_main1}
Let $b\in W^{1,1}(\T^d  )$ be weakly divergence free \eqref{divergence}. If $u\in \dot{L}^\infty$ is such that $\int (\nabla u + bu ) \cdot \nabla \phi =0$ for all $\phi\in C^\infty (\T^d)$, then $u=0$.
\end{theorem}
In particular, by Lemma \ref{zhikovs_lem1.5}, approximation solutions are unique in the class of weak solutions for $b\in W^{1,1}$. We note that in the Zhikov counterexample (Example~\ref{ex_zhikov}) the solution $u$ is bounded, and the drift $b$ just fails to belong to $W^{1,1}$. However, $f \not \in L^\infty$, and so it is not clear whether any approximation solution is bounded. Thus it is not clear whether Theorem \ref{thm_main1} is sharp in the sense that it is not clear whether for $b\not \in W^{1,1}$ bounded solutions can be nonunique. 

On the other hand, Theorem~\ref{thm_main1} is sharp in the sense that the regularity assumption $u\in \dot L^\infty $ cannot be replaced by $\dot L^p$ for any $p<\infty$ if the dimension $d>p+1$, see Corollary~\ref{cor_main3} below.

%{\rd Let us also emphasize that the zero divergence condition \eqref{divergence} seems necessary in our argument. Despite the fact that the DiPerna-Lions lemma still works in case of $b$ with its divergence bounded, we cannot conclude the duality argument in such a case.}  

 We prove Theorem~\ref{thm_main1} in Section~\ref{sec_UniS} below, using a duality method and a DiPerna-Lions-type commutator estimate \cite{DiPL}. In this proof the assumption $\div b =0$ is necessary.

Note that Theorem~\ref{thm_main1} still does not address the uniqueness problem for $b\in L^p$, $p\in (2d/(d+2),2)$ even among approximation solutions, which remains an open problem.

The second main result of this note shows that the problem admits a lot of flexibility for
\eqnb\label{p_for_nonuniqueness}
p<\frac{2(d-1)}{d+1} < \frac{2d}{d+2}.
\eqne
In fact, not only solutions are nonunique for some $b\in L^p$ for such $p$'s, but the set of such $b$'s is dense in $L^p$.
\begin{theorem}[Asymptotically sharp nonuniqueness]\label{thm_main2}
Let $d\geq 4$ and $p\in \left( 1,\frac{2(d-1)}{d+1} \right)$. Given $\epsilon>0$ and a divergence-free $b_0\in L^p $ there exists another divergence-free $b\in L^p $ such that $\| b-b_0 \|_p \leq \epsilon$ and \eqref{zhikov_weak_intro} with $f=0$ has a nontrivial weak solution $u\in H^1 \cap \dot L^{p'}$ with $\int u =0$.
\end{theorem}
As before, we use the convention $\int\equiv \int_{\T^d}$, $L^p\equiv L^p (\T^d )$, $H^1 \equiv H^1 (\T^d)$. We note that the range of $p$'s for $d=2,3$ is empty.

Even though the range \eqref{p_for_nonuniqueness} excludes the approximation solutions of \eqref{zhikov_weak_intro} (recall \eqref{b_approximate}), solutions in the class $H^1 \cap L^{p'}$ exist for any $f\in L^\infty$, as pointed out below \eqref{max_princ}. Thus Theorem~\ref{thm_main2} is asymptotically sharp in the sense that it shows optimality of the $L^2$ regularity of the drift $b$ for uniqueness, at least for sufficiently high dimension $d$. We note that in the above theorem we consider $u\in H^1$, since this is the class in which the uniqueness holds (for $b\in L^2$, recall \eqref{b_l2}) and also the energy inequality \eqref{en_ineq} is concerned with such $u$. In particular, Theorem~\ref{thm_main2} demonstrates that if the regularity of the drift $b$ is very low, then there exist a lot of weak solutions violating the energy inequality. 

Moreover, in the case when the $H^1$ regularity of $u$ is not required we obtain the following.

\begin{cor}[Nonuniqueness of less regular solutions]\label{cor_main3} 
Let $d\geq 3$ and $p\in (1, d-1 )$ and let $r\in [1,p'(d-1)/(d-1+p'))$.
Given $\epsilon >0$ and a divergence-free $b_0 \in  L^{p}$ there exists $b\in L^{p}$ such that $\div b =0$ with $\| b-b_0 \|_{L^p}\leq \epsilon $ such that $\int (-u \Delta \phi + bu \cdot \nabla \phi )=0$ for all $\phi \in \dot C^\infty$, for some nontrivial solution $u\in \dot L^{p'} \cap W^{1,r}$.

Moreover if also $p> (d-1)/(d-2)$ and $q\in [1, p(d-1)/(d-1+p))$ then ``$L^p$'' above can be replaced by ``$L^p\cap W^{1,q}$''. 
\end{cor}
Corollary~\ref{cor_main3} shows that the range of $p$ can be significantly expanded, if one does not require that $u\in H^1$. In particular, this shows that Zhikov's uniqueness result \eqref{b_l2} of weak solutions does not hold for less regular solutions, i.e. for $u\in \dot L^{p'} \cap W^{1,r}$. 

Moreover, for $p<2(d-1)/(d+1)$ one can take $r=2$ in Corollary~\ref{cor_main3} to recover Theorem~\ref{thm_main2}.

The last claim of the corollary shows that Theorem~\ref{thm_main1} is sharp for large dimension $d$. Indeed, given $p'<\infty$ taking any $d>p'+1$ ensures that the range of $q$ is nonempty. Then taking $q\coloneqq 1$ gives nonuniqueness of solutions $u\in \dot L^{p'}$ for $b\in W^{1,1}$. This shows that the uniqueness of solutions $u\in \dot L^\infty$, which is valid in all dimensions due to Theorem~\ref{thm_main1}, is sharp.\\

We prove Theorem~\ref{thm_main2} in Section~\ref{sec_Mikadoflows} below using a Nash iteration. The main idea of such iterations was first introduced in the groundbreaking work of Nash \cite{nash_54}, and it has been introduced to problems in partial differential equations by M\"{u}ller and \v{S}ver\'{a}k \cite{ms_03}, as well as De Lellis and Sz\'{e}kelyhidi Jr. \cite{DLS,DLS1} in the context of the Euler equations. The latter works inspired a number of groundbreaking developments using Nash iterations \cite{DanS,isett_onsager}, sometimes referred to as ``convex integration'', leading to the proof of the flexible side of the Onsager conjecture \cite{BDSV_19,isett_onsager}, as well as other remarkable results in different models \cite{bv_2019,cl_arxiv,MSa,MS}.

Our {proofs of Theorem~\ref{thm_main2} and Corollary~\ref{cor_main3} offer} an application of some of these developments in the Zhikov problem, and use Mikado flows, developed by \cite{DanS,MS}. They are  inspired by the approach of Modena and Sz\'{e}kelyhidi Jr. \cite{MS} in the context of the transport equation, but in our case the regularity of the constructed solution $u$ in Theorem \ref{thm_main2} is higher than the regularity of solutions to the transport equation obtained by Modena and Sz\'{e}kelyhidi Jr. \cite{MS}, the price we pay is the lower regularity of the constructed drift $b$. In fact, the roles of $u$ and $b$ in Theorem~\ref{thm_main2} can be thought of as opposite to the case of \cite{MS}. Moreover, our problem is time independent, which resembles a recent result \cite{luo} regarding existence of steady solutions to the $4$ dimensional Navier-Stokes equations, see also  \cite{bcv}, which exposes the flexibility related to the time dependence. 

\section{Proof of Theorem~\ref{thm_main1}}\label{sec_UniS}
Here we show that if $b\in W^{1,1}$ is divergence-free and $u\in \dot{H}^1 \cap L^\infty$ satisfies
\eqnb\label{eq_for_u_homo}
\int  (\nabla  u+  bu)\cdot \nabla \phi= 0
\eqne
for all $\phi \in C^\infty$, then $u=0$, which proves Theorem~\ref{thm_main1}.

Given $f\in \dot L^\infty (\T^d )$ let $v\in \dot H^1 (\T^d )\cap L^\infty (\T^d )$ be an approximation solution to \eqref{zhikov_intro} with $b$ replaced by $-b$, that is
\eqnb\label{eq_for_v}
\int \na v \cdot \na \phi  - \int bv\na \phi = \int f \phi\qquad \text{ for }\phi \in  C^\infty (\T^d).
\eqne
We denote by $\rho_\varepsilon (x) \coloneqq  \frac{1}{\varepsilon^n } \rho (x/\varepsilon )$ a standard mollifier, where $\rho \in C^\infty_0(B_1)$ is such that $\int_{B_1} \rho =1$. We set
\eqnb\label{mollification}
v_\varepsilon \coloneqq  v \ast \rho_\varepsilon,
\eqne
where ``$\ast$'' denotes the convolution. Using $v_\varepsilon $ as a test function in the equation for $u$ we get
\[\begin{split}
0&= \int (\nabla u + bu )\cdot \na v_\varepsilon  \\
&=  \int \na u_\varepsilon \cdot \na v - \int ((b u)\ast \na \rho_\varepsilon ) v\\
&= \int (b \cdot \na u_\varepsilon - (bu ) \ast \na \rho_\varepsilon )v + \int f u_\varepsilon \\
&= \int \int u (y) (b(x)-b(y)) \cdot \nabla \rho_\varepsilon (x-y)  v(x) \mathrm{d} y\,\mathrm{d} x +  \int f u_\varepsilon \\
&=- \int \int u (x-z\varepsilon ) \frac{b(x)-b(x-z\varepsilon )}{\varepsilon } \cdot \nabla \rho (z)  v(x) \mathrm{d} z\,\mathrm{d} x +  \int f u_\varepsilon \\
&\stackrel{\varepsilon \to 0}{\longrightarrow} - \int u (x) v(x) \p_i b (x )\mathrm{d} x \cdot \int  \nabla \rho (z)z_i \mathrm{d} z  +\int f u \\
&=\int f u ,
\end{split}
\]
where we used \eqref{eq_for_v} in the third line, the change of variable $y\mapsto (x-y)/\varepsilon =:z$ in the fifth line, and the fact that $\int z_j \p_i \rho (z) \d z = \delta_{ij}$, where $\delta_{ij}$ denotes the Kronecker delta, together with the divergence-free property of $b$ in the last line. In the 6th line we used the Dominated Convergence Theorem by observing that, since both $u$ and $v$ are essentially bounded, the integrand can be bounded by a constant multiple of $ \left| (b(x)-b(x-z\varepsilon ))/\varepsilon  \right|$, which in turn converges in $L^1 (\T^d\times \T^d)$ due to the fact that $b\in W^{1,1}$.

Taking $f= u$ gives $u=0$, as required.

\section{Proof of Theorem~\ref{thm_main2} and Corollary~\ref{cor_main3}}\label{sec_Mikadoflows}

Here we prove Theorem~\ref{thm_main2}, that is we show that for every $\epsilon>0$ and every divergence-free $b_0 \in C^\infty $ there exists a divergence-free $b\in L^p $ such that $\| b-b_0 \|_p \leq \epsilon$ and, for some nontrivial $u\in H^1 \cap L^{p'}$ with $\int u =0$,
\eqnb\label{zhikov_weak_repeat}
\int ( \nabla u + bu )\cdot \nabla \phi =0
\eqne
holds for all $\phi \in C^\infty $. Since $C^\infty$ is dense in $L^p$, this proves Theorem~\ref{thm_main2}.

The claim can be proved using the following.
\begin{proposition}\label{prop_the_step}
Let $M>0$ be the constant from Lemma~\ref{3.2} and $p\in (1, 2(d-1)/(d+1))$. Suppose that $(b_0 , u_0 , f_0 )\in C^\infty (\T^d ; \RR^{2d+1})$, with $\div b_0=0$, $\int u_0=0$ satisfies the equation
\[
-\div (\nabla u_0 +  b_0u_0 ) = \div f_0 \qquad \text{ in } \T^d.
\]
Given $\varepsilon >0$ there exists another triple $(b_1 , u_1 , f_1 )\in C^\infty (\T^d ; \RR^{2d+1})$, with $\div b_1=0$, $\int u_1=0$, satisfying the same equation and such that
\begin{eqnarray}
\| b_1 - b_0 \|_{p } + \| u_1 - u_0 \|_{{p'} } &\leq &M \max \{ \| f_0 \|_1^{ 1/p' }, \| f_0 \|_1^{1/p} \} \label{raz},\\
\| u_1-u_0 \|_{H^1} +  \| f_1 \|_{1} &\leq &\varepsilon \label{dwa}.
\end{eqnarray}
\end{proposition}

\begin{proof}[Proof of Theorem \ref{thm_main2} using Proposition \ref{prop_the_step}.] Given $b_0$ and $\epsilon$ we can pick any nontrivial $u_0 \in C^\infty (\T^d)$ such that $\int u_0 =0$ and we set
\[ f_0 \coloneqq  - \nabla u_0 - b_0 u_0. \]
Since $f_0$ depends linearly on $u_0$, we can  assume (by multiplying $u_0$ by a small number) that
\begin{equation}\label{choice_of_u0}
\max \left\lbrace \| f_0 \|_1^{ 1/p' }, \| f_0 \|_1^{1/p} \right\rbrace \leq \frac{\epsilon }{2M}.
\end{equation}

For $k\geq 1$ we apply Proposition~\ref{prop_the_step} with
\eqnb\label{choice_vareps} \varepsilon \coloneqq  \frac12 \min \left\lbrace 1, \left( \frac{\epsilon }{M\,2^{k+1} } \right)^{p'} ,  2^{-k} \| u_0 \|_{H^1} \right\rbrace
\eqne
to obtain $(b_k,u_k,f_k)$ satisfying
\eqnb\label{eq_for_uk_bk}
-\div (\na u_k + b_ku_k ) =\div f_k.
\eqne
 Note that, due to \eqref{raz}, \eqref{dwa}, this implies that $(b_k)$ is Cauchy in $L^p$, and $(u_k)$ is Cauchy in $L^{p'}$ and also Cauchy in $H^1$. Thus there exist a divergence-free $b\in L^p$ such that $b_k \to b $ in $L^p$, as well as $u\in H^1 \cap L^{p'}$ such that $\int_{\T^d} u =0$ and $u_k \to u$ in $H^1$ and in $L^{p'}$. That $b$, $u$ satisfy \eqref{zhikov_weak_repeat} follows by taking the limit in the weak formulation of \eqref{eq_for_uk_bk}. Furthermore, due to our choice of $\varepsilon $ and \eqref{raz}, \eqref{dwa},
\[
\| u_k \|_{H^1} \geq \| u_0 \|_{H^1} - \sum_{l=1}^k \| u_l - u_{l-1} \|_{H^1} \geq \| u_0 \|_{H^1} \left( 1 - \sum_{l=1}^k 2^{-{(l+1)}}\right) \geq \frac{\| u_0 \|_{H^1}}2 \qquad \text{ for }k\geq 1,
\]
which shows that $u\ne 0$, and
\[
\| b - b_0 \|_p \leq \| b_1 - b_{0} \|_p +\sum_{k\geq 2} \| b_k - b_{k-1} \|_p \leq \frac{\epsilon}2 + \sum_{k\geq 2} M \| f_{k-1} \|_1^{1/p'}\leq \frac{\epsilon}2 + \sum_{k\geq 2} \frac{\epsilon }{2^{k}}  = \epsilon,
\]
where we used \eqref{choice_of_u0} to obtain the first $\epsilon/2$. Thus $\|b - b_0 \|_{p} \leq \epsilon$, as required.
\end{proof}

It remains to prove Proposition~\ref{prop_the_step}. To this end we recall the Mikado flows on $\T^d$.
\begin{lemma}[Mikado flows]\label{3.2}
There exists $M>0$ with the following property. Given $\mu > 2d $ and $j\in \{ 1,\ldots , d \}$ there exists a \emph{Mikado density} $\Theta_\mu^j: \T^d \rightarrow \R$ and a \emph{Mikado field} $W_\mu^j: \T^d \rightarrow \R^d$ such that
\eqnb \label{div_and_cancel_props}
\begin{split}
\div\, W_\mu^j =\div \left( \Theta_\mu^j W_\mu^j \right)&=0,\\
\int \Theta_\mu^j = \int W_\mu^j &=0 ,\\
\int \Theta_\mu^j W_\mu^j &= e_j
\end{split}
\eqne
for every $k\geq 0$ and
\eqnb\label{mikado_bds}
\begin{split}
\sum_{j=1}^d \| \na^k \Theta_\mu^j \|_r &\leq \frac{M}3 \mu^{k+(d-1)\left( \frac{1}{p'}-\frac{1}{r}  \right)} ,\\
\sum_{j=1}^d \| \na^k W_\mu^j \|_r &\leq  \frac{M}3 \mu^{k+(d-1)\left( \frac{1}{p} -\frac{1}{r}  \right)} ,\\
\sum_{j=1}^d \| \Theta_\mu^j  W_\mu^j \|_1 &\leq  M,\\
\sum_{j=1}^d \|  \Theta_\mu^j \|_{H^1} &\leq  {M}\mu^{-\gamma} ,
\end{split}
\eqne
where $\gamma \coloneqq   {(d-1)\left(\frac{1}{p} + \frac{1}{2} - (1+\frac{1}{d-1} )\right)}>0$, since $p\in\left( 1, \frac{2(d-1)}{d+1}\right) $.
\end{lemma}
The Mikado flows have been first introduced by Daneri and Sz\'ekelyhidi Jr. \cite{DanS} in the context of the Onsager conjecture regarding the incompressible Euler equations. Here we use the Mikado flows that were applied by Modena and Sz\'ekelyhidi Jr. \cite{MS} in the context of the transport equation, see Lemma~2.1 and Lemma~2.6 in \cite{MS} for details.

Given $f: \T^d \to \R$ and $\lambda>0$, we denote by
\[
f_\lambda (x) \coloneqq  f(\lambda x)
\]
the ``$\lambda$-dilation of $f$''. We now recall some facts regarding functions with fast oscillations.
\begin{lemma}[Fast-oscillating functions]\label{lem_fast_osc}
Let $\lambda \in \mathbb{N}$.
\begin{enumerate}
\item[1.] (improved H\"older's inequality) For all $p\in [1,\infty ]$ there exists $C_p>0$ such that
\begin{equation}\label{improved_Holder}
\left| \| f g_\lambda \|_p - \| f \|_p \|g \|_p \right| \leq C_p \lambda^{-1/p} \| f \|_{C^1} \| g \|_p
\end{equation}
for all $f,g \in C^\infty $.
\item[2.] (quantitative Riemann-Lebesgue) If $\int g =0$ then
\begin{equation}\label{riem-leb_quantified}
\left|\int f g_\lambda  \right| \leq \sqrt{d} \lambda^{-1} \| f \|_{C^1} \| g \|_1.
\end{equation}
\item[3.] (the antidivergence operator $\mathcal{R}$) Given $f,g \in C^\infty $ with $\int fg_\lambda = \int g =0$ there exists a vector field $u=\mathcal{R} (fg_\lambda )$ such that $\div\, u=fg_\lambda $ and
\begin{equation}\label{antidiv_bound}
\| \na^k u \|_p \leq C_{k,p} \lambda^{k-1} \| f\|_{C^{k+1}} \| g \|_{W^{k,p}}
\end{equation}
for $k\geq 0$, $p\in [1,\infty ]$, where $C_{k,p}>0$.
\end{enumerate}
\end{lemma}
Recall that we use the convention $\int\equiv \int_{\T^d}$ and that all norms and function spaces are considered on $\T^d$. The proof of the lemma follows from Lemmas~2.1-2.6 in \cite{MS}. We can now prove the above proposition, where we will write $f\equiv f_0$ for brevity.
\begin{proof}[Proof of Proposition~\ref{prop_the_step}.] Let $\delta >0$ and let $\chi_j \in C^\infty $ be such that $\chi_j =1$ on $\{ |f_j (x) |\geq \delta /2d \}$ and $\chi_j=0$ on $\{ |f_j (x) | \leq \delta /4n \}$. Let
\[
\begin{split}
\theta (x) &\coloneqq   \sum_{j=1}^n \chi_j (x) \mathrm{sgn} ( f_j (x) ) |f_j (x)|^{1/p'} \left( \Theta_\mu^j\right)_\lambda , \\
\theta_c  &\coloneqq  -\int \theta , \\
w (x) &\coloneqq  \sum_{j=1}^n \chi_j (x)  |f_j (x)|^{1/p} \left( W_\mu^j\right)_\lambda , \\
w_c (x) &\coloneqq  - \sum_{j=1}^n \mathcal{R} \left( \na \left( \chi_j (x) |f_j (x)|^{1/p} \right) \cdot \left( W_\mu^j\right)_\lambda \right),
\end{split}
\]
where, in the definition of $w_c$, we have used the facts that $\int W_\mu^j =0$ and 
\[
\int\na \left( \chi_j (x) |f_j (x)|^{1/p} \right) \cdot \left(  W_\mu^j \right)_\lambda =\int \div\left(\chi_j (x) |f_j (x)|^{1/p}\left(  W_\mu^j \right)_\lambda\right)=0
\]
(see \eqref{div_and_cancel_props}) to use the antidivergence operator $\mathcal{R}$ introduced in \eqref{antidiv_bound}. We set
\eqnb\label{def_of_u1_b1}
\begin{split}
u_1 &\coloneqq u_0 + \theta + \theta_c,\\
b_1 &\coloneqq b_0 + w + w_c.
\end{split}
\eqne
Note that $u_1,b_1\in C^\infty $ and $\div b_1=0$, $\int u_1=0$. Moreover, since $\div (\na u_0 + b_0 u_0 ) = -\div f$,
\begin{equation}\label{eq_for_u1b1}
\begin{split}
\div (\na u_1 + b_1 u_1 ) &= \div \left( \underbrace{\theta w - f}_{=: g_{main}} + \underbrace{\na \theta }_{=: g_{Lapl}}  + \underbrace{w u_0 + b_0 \theta}_{=: g_{lin}} + \underbrace{w_c u_0 + b_0 \theta_c + \theta_c w + \theta w_c  + \theta_c w_c}_{=: g_{corr}} \right) \\
&=: \div g.
\end{split}
\end{equation}
In what follows we will decompose $g_{main}$ into $g_{quad} +g_{\chi}$ with $\| g_{\chi }\|_1 \leq \delta /2$ and we will obtain the following estimates
\eqnb\label{ests_to_show}
\begin{split}
\| \theta \|_{H^1} &\leq    C\,\lambda^{1/2} \mu^{-\gamma } ,\\
| \theta_c | &\leq C\,\lambda^{-1} ,\\
\| w \|_p + \| w_c \|_p &\leq \frac{M}3 \| f \|_1^{1/p}  + C \left( \lambda^{-1/p} +\lambda^{-1}\right), \\
\| \theta  + \theta_c \|_{p'} &\leq \frac{M}3 \| f \|_1^{1/p'} + C\, \lambda^{-1/p'}, \\
\| g_{quad} \|_1 + \| g_{Lapl} \|_1 +\| g_{lin} \|_1 +\| g_{corr} \|_1  &\leq  C\left( \lambda^{-1} + \mu^{-(d-1)/p } + \mu^{-(d-1)/p'} \right),
\end{split}
\eqne
where $C=C(d,p,M,f,u_0,b_0,\delta )>0$. The claim of Proposition~\ref{prop_the_step} then follows by first taking $\delta >0$ sufficently small, then taking $\lambda$ sufficiently large, and finally taking $\mu >0$ sufficiently large.\\

In order to prove \eqref{ests_to_show}, we first note that
\eqnb\label{est_w_Lp}
\begin{split}
\| w  \|_{p} &\leq \sum_{j=1}^d \left\| \chi_j  |f_j |^{1/p} \left( W_\mu^j\right)_\lambda \right\|_{p} \\
&\leq   \sum_{j=1}^d \left( \left\| \chi_j   |f_j |^{1/p} \right\|_{p} \left\|  W_\mu^j \right\|_{p}  + C \lambda^{-1/p} \left\| \chi_j (x) |f_j |^{1/p} \right\|_{C^1}   \left\| W_\mu^j \right\|_{p}  \right) \\
& \leq   \frac{M}3 \| f \|_1^{1/p}  + C \lambda^{-1/p},
\end{split}
\eqne
where we used the improved H\"older inequality \eqref{improved_Holder} in the second inequality, as well as \eqref{mikado_bds} with $k=0$ and our choice of $\chi_j$ in the last inequality. In order to estimate $w_c$ we recall (from Lemma~\ref{3.2}) that $\int W_\mu^j =0$ and that
\[
\int\na \left( \chi_j (x) |f_j (x)|^{1/p} \right) \cdot W_\mu^j=\int \div\left(\chi_j (x) |f_j (x)|^{1/p}W_\mu^j\right)=0.
\]
Thus we can use \eqref{antidiv_bound} to obtain
\begin{equation}\label{est_wc_Lp}
\| w_c \|_p \leq  C\, \lambda^{-1} \sum_{j=1}^d \| W_\mu^j \|_p\leq C\, \lambda^{-1} .
\end{equation}
Analogously to \eqref{est_w_Lp} we obtain that
\begin{equation}\label{est_theta_Lp'}
\| \theta  \|_{p'} \leq   \frac{M}3 \| f \|_1^{1/p'}  +  C \, \lambda^{-1/p'}.
\end{equation}
As for $\theta_c$ we use the quantitative Riemann-Lebesgue \eqref{riem-leb_quantified} and \eqref{mikado_bds} to obtain
\begin{equation}\label{est_theta_c}
|\theta_c | \leq  \sum_{j=1}^n \sqrt{n} \lambda^{-1} \| f \|_{C^1} \|  \Theta_\mu^j \|_1 \leq C\, \lambda^{-1}.
\end{equation}
As for $\na \theta$, we have
\[
\na \theta =  \sum_{j=1}^n \left( \na \left( \chi_j  \,\mathrm{sgn} ( f_j  ) |f_j|^{1/p'} \right) \left( \Theta_\mu^j\right)_\lambda+\chi_j  \,\mathrm{sgn} ( f_j ) |f_j |^{1/p'} \lambda \left( \na \Theta_\mu^j\right)_\lambda \right),
\]
and so
\[
\| \na \theta \|_2 \leq C\, \sum_{j=1}^d \left( \| \Theta_\mu^j \|_2 + \lambda^{1/2}\| \na \Theta_\mu^j \|_2 \right) \leq C\,\lambda^{\frac12} \mu^{-\gamma },
\]
where we used the improved H\"older inequality \eqref{improved_Holder} in the first inequality and the last estimate in \eqref{mikado_bds} in the second inequality. A similar estimate for $\| \theta \|_2$ follows, and so
\begin{equation}\label{est_theta_H1}
\| \theta \|_{H^1} \leq C\, \lambda \mu^{-\gamma }.
\end{equation}
Note that
\[
g_{main} = \sum_{j=1}^n \chi_j f_j \left( \Theta_\mu^j W_\mu^j \right)_{\lambda } -f =  \sum_{j=1}^n  \chi_j f_j \left(  \Theta_\mu^j W_\mu^j  -e_j \right)_{\lambda } + \underbrace{ \sum_{j=1}^n (1-\chi_j ) f_j e_j }_{=: g_\chi },
\]
and so
\[
\div g_{main} = \sum_{j=1}^n \na \left(  \chi_j f_j \right) \cdot\left(  \Theta_\mu^j W_\mu^j  -e_j \right)_{\lambda } + \div g_{\chi} .
\]
Noting that \eqref{div_and_cancel_props} implies that both the mean of $\Theta_\mu^jW_\mu^j-e_j$ vanishes and that
\[
\int \na \left(  \chi_j f_j \right) \cdot  \left(  \Theta_\mu^j W_\mu^j  -e_j \right)_{\lambda } = \int \div \left(  \chi_j f_j  \left(  \Theta_\mu^j W_\mu^j  -e_j \right)_{\lambda } \right)  = 0,
\]
we can use the antidivergence operator introduced in Lemma~\ref{lem_fast_osc}.3 to define 
\[g_{quad }\coloneqq  \sum_{j=1}^n \mathcal{R} \left( \na \left(  \chi_j f_j \right) \cdot \left(  \Theta_\mu^j W_\mu^j  -e_j \right)_{\lambda } \right).
\]
We thus have
\[
\div g_{main} = \div g_{quad} + \div g_{\chi} ,
\]
and so we can replace $g_{main}$ by $g_{quad} + g_{\chi}$ (as then \eqref{eq_for_u1b1} remains valid). Using \eqref{antidiv_bound} with $k=0$ we obtain
\[
\| g_{quad} \|_{1} \leq C\,\lambda^{-1} \sum_{j=1}^n \| \chi_j f_j \|_{C^2} \| \Theta_\mu^j W_\mu^j - e_j \|_1  \leq  C\, \lambda^{-1},
\]
where we also used \eqref{mikado_bds} and our choice of $\chi_j$'s in the last inequality. Moreover $(1-\chi_j) f_j$ vanishes when $|f_j|\geq \delta /2d$, so that
\[
\| g_\chi \|_1 \leq \sum_{j=1}^n \| (1-\chi_j ) f_j \|_1 \leq  \sum_{j=1}^n  \int_{|f_j|< \delta/2d } |f_j | \leq \frac{\delta }2,
\]
as required. By \eqref{est_theta_H1} we have $\| g_{Lapl} \|_1 \leq C (\delta , \| f \|_{C^1} ) M \lambda \mu^{-\gamma }$. As for $g_{lin}$ we can use the first two estimates in \eqref{mikado_bds} with $k\coloneqq  0$, $r\coloneqq  1$ to obtain
\[
\| g_{lin} \|_1 \leq \| w \|_1 \| u_0 \|_{\infty } + \| b_0 \|_{\infty } \| \theta \|_{1} \leq C\, ( \mu^{-(d-1)/p} + \mu^{-(d-1)/p'} ).
\]
Finally,
\[
\| g_{corr} \|_1 \leq  \| w_c \|_p \| u_0 \|_{p'} + \| \theta \|_{p'} \| w_c \|_p + |\theta_c | ( \| b_0 \|_1 + \| w \|_1 + \| w_c \|_1 )\leq  C  \left( \lambda^{-1} + \mu^{-(d-1)/p'}\right) ,
\]
where we used \eqref{est_wc_Lp}, \eqref{est_theta_c}, \eqref{est_theta_Lp'} and the second estimate in \eqref{mikado_bds} with $k\coloneqq   0$, $r\coloneqq  1$ in the last inequality. Thus we have obtained all inequalities claimed in \eqref{ests_to_show}.
\end{proof}

In order to prove Corollary~\ref{cor_main3} we revisit the proof above and note that the restriction on the range of $p$ originates from the last estimate of \eqref{mikado_bds}, that is from the requirement that $u\in  H^1$.  Dropping this requirement we may instead require that $u\in W^{1,r}$ or impose additional regularity on $b$, as much as the first two estimates of \eqref{mikado_bds} allow. To be more precise, the first claim of Corollary~\ref{cor_main3} follows from a version of Proposition~\ref{prop_the_step} in which the iterative estimates \eqref{raz}--\eqref{dwa} are replaced by 
\begin{eqnarray}
\| b_1 - b_0 \|_{p } + \| u_1 - u_0 \|_{{p'} } &\leq &M \max \{ \| f_0 \|_1^{ 1/p' }, \| f_0 \|_1^{1/p} \} \label{raz1},\\
\| u_1-u_0 \|_{W^{1,r}} +  \| f_1 \|_{1} &\leq &\varepsilon \label{dwa1},
\end{eqnarray}
and with the $H^1$ norm replaced by the $W^{1,r}$ norm in \eqref{choice_vareps} and in the proof of $u\ne 0$ following it.

In order to see \eqref{raz1}--\eqref{dwa1} one needs to replace the $\| \na \theta \|_2$ estimate \eqref{est_theta_H1} by 
\[
\| \nabla \theta \|_r \leq C_r\sum_{j=1}^d \left( \| \Theta_\mu^j \|_r + \lambda^{1-\frac{1}{r}}\| \na \Theta_\mu^j \|_r \right) \leq C_r\lambda^{1-\frac{1}{r}} \mu^{1+(d-1)\left( \frac{1}{p'} - \frac{1}{r} \right)},
\]
using \eqref{improved_Holder} and \eqref{mikado_bds}. This can be made small, as before, as the exponent of $\mu$ is negative. We note that this estimate needs to be used in the estimate for $\| g_{Lapl}\|_1 $. This guarantees \eqref{dwa1}, while \eqref{raz1} follows in the same way as before.

Analogously, one can obtain the second claim of Corollary~\ref{cor_main3} by showing that, except for \eqref{dwa1} we also have
\eqnb\label{dwa2}
\|b_1 - b_0 \|_{W^{1,q}} \leq \varepsilon.
\eqne
Indeed, using \eqref{improved_Holder},  \eqref{antidiv_bound} and \eqref{mikado_bds},
\[
\begin{split}\| \na w \|_q &\leq C_q \sum_{j=1}^d \left( \| W_\mu^j \|_q + \lambda^{1-\frac1{q} } \| \na W_\mu^j \|_q  \right)  \leq C_q \lambda^{1-\frac1{q} } \mu^{1+(d-1) \left(\frac{1}{p} - \frac{1}{q}\right) }\\
\text{ and }\,\,\| \na w_c \|_q &\leq C_q \sum_{j=1}^d \| W_{\mu}^j \|_q \leq C_q \mu^{(d-1)\left(\frac{1}{p}-\frac{1}{q} \right)},
\end{split} 
\]
which can be both made smaller than $\varepsilon$ as above.\vspace{0.5cm}\\

\noindent
\textbf{Acknowledgements.} We are grateful to an anonymous referee for many detailed and insightful comments. T.C. is grateful to A. S. Mikhailov from the St. Petersburg branch of Steklov Institute for introducing him to the problem and interesting discussions. We are grateful to T. Komorowski from IMPAN for pointing out the potential role of Sobolev regularity of $b$, and to M. Ma\l{}ogrosz for helpful discussions.

 T.C. was supported by the National Science Centre (NCN) grant SONATA BIS 7 UMO-2017/26/E/ST1/00989. W.S.O. was supported in part by the Simons Foundation.\\

\noindent
\textbf{Data availability statement.} 
There is no data associated with the project.

\appendix
\section*{Appendix. Proof of the maximum principle \eqref{max_princ}}\label{App}

Here we show that, if $f\in  L^\infty $ and $b\in L^2$ satisfies the divergence-free condition \eqref{divergence}, then any solution $u\in L^\infty\cap  \dot{H}^1$ to $-\div (\nabla u + bu ) =f$ (in the sense of  \eqref{zhikov_weak_intro}) satisfies
\eqnb\label{max_princ_toshow}
\| u \|_{\infty } \leq C_n \| f \|_\infty,
\eqne
where $C_n>0$ depends only on the dimension $d$. Note that, since the constant does not depend on $b$, this proves \eqref{max_princ_toshow} for all $u\in \dot H^1$ (i.e. we can obtain $u\in L^\infty $, rather than assume it), which is \eqref{max_princ}, as desired. Indeed, one can approximate any $b\in L^{2d/(d+2)}$ in the $L^{2d/(d+2)}$ norm by a $C^\infty$ function and also approximate $f$ in $H^{-1}$ by a smooth function, which gives a unique and smooth solution the approximate system. The limiting procedure described below \eqref{b_approximate} then gives an approximation solution $u$ with \eqref{max_princ_toshow}. Since for $b\in L^2$ the uniqueness of solutions to \eqref{zhikov_weak_intro} holds (recall \eqref{b_l2}), we obtain \eqref{max_princ}.

We first recall the inequality
\eqnb\label{gn_to_use}
\| g \|_2^2 \leq \varepsilon \| \nabla g \|^2_2 + C_n \varepsilon^{-\frac{n}2} \| g \|_1^2,
\eqne
valid for every $g\in H^1 (\T^d)$, $\varepsilon \in (0,1)$.  Indeed the Gagliardo-Nirenberg-Sobolev inequality gives that
\[
\| g \|^2_2 \leq C \| \nabla g \|_2^{\frac{2d}{d+2}} \| g \|_1^{\frac4{d+2}} \leq \varepsilon \| \nabla g \|_2^2 + C_d \varepsilon^{-\frac{d}2} \| g \|_1^2
\]
if $\int g =0$, where we used Young's inequality in the second step. Thus the claim follows for such $g$. If $\int g \ne 0$ then
\[
\begin{split}
\| g \|_2^2 &\leq C \left\| g- \int g \right\|_2^2 + C\left| \int g \right|^2 \\
&\leq \varepsilon \| \nabla g \|_2^2 + C_n \varepsilon^{-\frac{n}2} \left\| g - \int g\right\|_1^2 + C\| g \|_1^2 \\
&\leq \varepsilon \| \nabla g \|_2^2 + C_n \varepsilon^{-\frac{n}2} \left\| g \right\|_1^2 ,
\end{split}
\]
where we used the triangle inequality and the assumption that $\varepsilon <1$ in the last step. \\

We can now prove \eqref{max_princ_toshow}. Applying \eqref{gn_to_use} with $g\coloneqq  {u^{2^{k-1}}}$ gives
\eqnb\label{temp1}
\int u^{2^k} \leq \varepsilon \int \left| \nabla u^{2^{k-1}} \right|^2 + C\varepsilon^{-\frac{n}2}  \left( \int u^{2^{k-1}}\right)^2,
\eqne
where $C>1$ is a generic constant, which depends on $d$ only. The value of $C$ may change from line to line in the following calculation.

Note that, since we assume $b\in L^2 $, we can test \eqref{zhikov_weak_intro} with $u^{2^k-1}$. Indeed, one can, for example, take $\phi \coloneqq  (u^{2^k-1})_\varepsilon$ (the mollification of $u$, recall \eqref{mollification}) and take the limit $\varepsilon\to 0$. We obtain
\eqnb\label{to_use_later}
\frac{2^k -1}{2^{2k-2}}  \int \left| \nabla u^{2^{k-1}} \right|^2 = \int \nabla u \cdot \nabla \left( u^{2^k-1} \right) =  \int f u^{2^k-1} \leq \frac12 \int f^{2^k} + \frac12 \int u^{2^k},
\eqne
where we used the fact that $\int bu \cdot \nabla \left( u^{2^k-1} \right) =(1-2^{-k}) \int b \cdot \nabla \left( u^{2^k} \right) =0$, due to the divergence-free assumption on $b$ and the assumed regularity $u\in L^\infty \cap H^1$, $b\in L^2$. Applying this inequality in \eqref{temp1} gives
\[
\int u^{2^k} \leq \varepsilon \frac{2^{2k-3}}{2^k-1} \int f^{2^k}  + \varepsilon \frac{2^{2k-3}}{2^k-1} \int u^{2^k} +  C \varepsilon^{-\frac{n}2} \left( \int u^{2^{k-1}} \right)^2.
\]
Taking $\varepsilon \coloneqq   2^{-k}$ and noting that
\[
2^{-k} \frac{2^{2k-3}}{2^k-1} \leq \frac14 \qquad \text{ for } k\geq 1,
\]
we obtain
\[
\int u^{2^k} \leq \frac14 \int f^{2^k}  + \frac14   \int u^{2^k} +  C 2^{\frac{kn}2} \left( \int u^{2^{k-1}} \right)^2 ,
\]
and so
\[
\int u^{2^k} \leq \int f^{2^k}  +  C 2^{\frac{kn}2} \left( \int u^{2^{k-1}} \right)^2  \leq   \| f\|_{\infty }^{2^k}  +  C 2^{\frac{kn}2} \| u \|_{2^{k-1}}^{2^k}
\]
for $k\geq 1$. Taking both sides to power $2^{-k}$ gives
\[
\| u \|_{2^k} \leq  \left(  \| f \|_{\infty }^{2^k} +  C 2^{\frac{kn}2}  \| u \|_{2^{k-1}}^{2^k}  \right)^{2^{-k}}
\]
for $k\geq 1$. Noting that the above inequality holds trivially when the left-hand side is replaced by $\| f \|_\infty$, we set
\[
m_k \coloneqq   \max \left\lbrace \| f \|_\infty , \| u\|_{2^k}   \right\rbrace
\]
and obtain that
\[
m_k \leq\left( 1+C 2^{\frac{kn}2} \right)^{2^{-k}}  m_{k-1} \leq \left( C 2^{\frac{kn}2} \right)^{2^{-k}}  m_{k-1}
\]
for $k\geq 1$. (Recall that the value of $C>1$ may change from line to line.) Since
\[
\prod_{k=1}^\infty \left( C 2^{\frac{kn}2} \right)^{2^{-k}} = \exp \left( \log \left( \prod_{k=1}^\infty \left( C 2^{\frac{kn}2} \right)^{2^{-k}} \right) \right) = \exp \left( \log C\sum_{k=1}^\infty 2^{-k} + \log \left( 2^{\frac{n}2} \right) \sum_{k=1}^\infty k2^{-k}  \right) \leq C,
\]
we see that
\[
m_k \leq C m_0\qquad \text{ for } k\geq 1.
\]
In particular $\| u \|_{2^k } \leq C m_0$ for $k\geq 1$, which implies that $u \in L^\infty$ with
\eqnb\label{almost_there}
\| u \|_\infty \leq C m_0.
\eqne
On the other hand the Poincar\'e inequality and \eqref{to_use_later} applied with $k=1$ gives
\[
\| u \|_2^2 \leq C \| \nabla u \|_2^2 =C  \int fu  \leq \frac12 \| u \|_2^2 + C \| f \|_\infty^2,
\]
which, after absorbing the first term on the right-hand side, implies that $m_0 \leq C \| f \|_\infty$. Applying this in \eqref{almost_there} gives \eqref{max_princ_toshow}, as required.

\end{document}